\date{October  7, 2019}

\documentclass[12pt]{amsart}
\usepackage{latexsym,amsmath,amsfonts,amscd,amssymb}
\usepackage{graphics}
\usepackage{xcolor}
\newtheorem{theorem}{Theorem}[section]
\newtheorem{lemma}[theorem]{Lemma}
\newtheorem{proposition}[theorem]{Proposition}
\newtheorem{corollary}[theorem]{Corollary}

\newtheorem{definition}[theorem]{Definition}

\theoremstyle{remark}

\newcommand{\Res}{\text{Res}}

\newcommand{\cO}{{\mathcal O}}

\newcommand{\CC}{{\mathbb C}}

\newcommand{\NN}{{\mathbb N}}

\newcommand{\RR}{{\mathbb R}}

\newcommand{\ZZ}{{\mathbb Z}}

\renewcommand{\Re}{\operatorname{Re}}
\renewcommand{\Im}{\operatorname{Im}}

\renewcommand{\a}{\alpha}

\newcommand{\g}{\gamma}

\title{Difference equations and Omega functions}


\author[R. P\'{e}rez-Marco]{Ricardo P\'{e}rez-Marco}
\address{CNRS, IMJ-PRG, Universit\'e Paris Cit\'e, Paris, France}
\email{ricardo.perez.marco@gmail.com}

\begin{document}

\maketitle

\begin{abstract}
We introduce Omega functions that generalize Euler Gamma functions and study the functional difference equation 
they satisfy. Under a natural exponential growth condition, the vector space of meromorphic solutions of the 
functional equation is finite dimensional. We construct a basis of the space of solutions composed by Omega functions. 
Omega functions are defined as exponential periods. They have a meromorphic
extension to the complex plane of order $1$ with simple poles at negative integers. Thevector space they span is 
characterized by their functional equation and 
their growth property on vertical strips. This generalizes Wielandt's 
characterization of Euler Gamma function. We also introduce Incomplete Omega functions that play an important 
role in the proofs.
\end{abstract}

\section{Introduction}

\subsection{Difference equations.}
We study in this article difference equations of the form
\begin{equation}\label{eq:functional1}
sf(s)=\sum_{k=1}^d \alpha_k f(s+k) 
\end{equation}
where $\a_1,\ldots , \a_d \in \CC$ and $\a_d\not=0$. The simplest case is 
the functional equation satisfied by Euler Gamma function
$$
s\Gamma(s) = \Gamma(s+1) 
$$

These equations are linear and we have a vector space of meromorphic solutions.
A natural motivation for studying these functional equations comes from the study of subspaces generated 
by natural linear operators. For instance, we can consider, in the space of meromorphic functions,
the shift (or integer translation) linear operator
$$
T(f(s))=f(s+1)
$$
and the multiplication by $s$ linear operator
$$
S(f(s))=sf(s) 
$$
Observe that $S$ has no eigenvectors and  
the minimal invariant subspace invariant by $S$ containing the constant functions is the space of polynomials $\CC[s]$. 
The space generated by $S$ and the function $f$ is the vector space $\CC[s] f$
$$
\langle f, S(f), S^2(f), \ldots \rangle = \CC[s] f
$$
It is natural to find the functions $f$ such that the space $\CC[s]f$ is generated by $f$ and $T$. 
This happens if and only if $f$ is a solution of the difference equation (\ref{eq:functional1}).

Already, in the simplest case of the functional equation of Euler Gamma function,
the space of solutions is infinite dimensional since any function of the form $e^{2\pi i n s}\Gamma(s)$ for 
an integer $n\in \ZZ$ is also a solution. 
It is classical to add conditions to characterize Euler Gamma function as the only 
normalized solution to this functional equation. One can mention Weierstrass characterization imposing some asymptotic 
behavior when $s\to +\infty$ (1856, \cite{Wei}), or Wielandt's 
characterization (1939, \cite{Wie}, see also \cite{Rem1}, \cite{Rem2}) 
requiring boundedness on vertical strips of width larger than $1$, or, more recently, requiring finite 
order of  the solutions and a right half plane free of zeroes nor poles (2022, \cite{PM1}). Wielandt's boundedness
condition has been weakened by Fuglede to a moderate growth in the vertical strip (2008, \cite{Fu}).

In the spirit of Wielandt, we search for solutions with some growth control on vertical strips. 
Under a suitable growth condition, we prove that the space of solutions is finite dimensional:

\begin{theorem}
The space of meromorphic solutions $f$ of the functional equation 
\begin{equation*}
sf(s)=\sum_{k=1}^d \a_k f(s+k) 
\end{equation*}
where $\a_1,\ldots , \a_d \in \CC$, $\a_d\not= 0$, and $f$ satisfies a growth condition, for $1\leq \Re s\leq d$,
$$
|f(s)|\leq Ce^{-c\Im s}
$$
for some constant $C>0$ and $0\leq c<2\pi$, is finite dimensional of dimension $d$.
\end{theorem}

Moreover, we build an explicit basis of the vector space of solutions with  
Special Functions, that we call Omega functions, that generalize Euler Gamma function.  
There is a large classical literature on 
linear difference equations with polynomial coefficients by Poincar\'e \cite{Poi}, Birkhoff \cite{Bi}, 
Carmichael \cite{Ca}, N\"orlund \cite{No}, and, more recently, 
solutions with vertical exponential growth have been studied by Barkatou \cite{Ba} and 
Duval \cite{Du} following work of Ramis. Even as early as 1739, Euler studied these difference equations with linear 
coefficients in \cite{[E123]}. The analysis 
of the functional equation in this article is self-contained and independent of the classical theory.


\subsection{Omega functions.}
Historically, Euler Gamma function appears for the first time in a letter from Euler  
to  Goldbach, dated January 8th 1730 (\cite{[E00717]}).  
Euler defines the Gamma function for real values $s>0$, by the 
integral formula
$$
\Gamma (s) =\int_0^{+\infty} t^{s-1} e^{-t} dt \ .
$$
which is also convergent for complex values of $s$ with $\Re s >0$. 
In this integral formula, the value $\Gamma(s)$ appears as an \textit{exponential period}. 

\textit{Algebraic periods} are integrals of algebraic differential forms over cycles of an algebraic 
variety. In the special case of an algebraic curve, when we represent the curve as a Riemann domain over 
the complex plane or the Riemann sphere, algebraic periods are also the integrals of algebraic differential 
forms on paths joining two ramification points where we have singularities of the differential form. 
From the transalgebraic point of view, it is natural to consider exponential periods, 
where integrals involve exponential expressions, and the singularities can be exponential singularities. 
More general periods can be envisioned where the differential form has transcendental singularities with 
monodromy like $t^s$ in a local variable (geometrically these correspond to differential forms living in a
branched Riemann domain with an infinite ramification). There is a vast literature on classical algebraic 
periods, but almost none on the transalgebraic periods.
We refer to \cite{KZ} for a survey about classical periods, and to \cite{BPM1} and \cite{BPM3} for exponential periods and 
their relation with log-Riemann surfaces. 
Also we refer the reader to \cite{PM1}  for a historical
survey of different definitions of Euler Gamma function and their generalizations, and to \cite{WW} for 
its classical properties.

We introduce (resp. Incomplete) Omega functions which are a 
natural generalization of the (resp. Incomplete) Gamma function. They are defined as:
$$
\Omega_k(s)=\int_0^{+\infty.\omega_k} t^{s-1} e^{P_0(t)} \, dt \ \ , \ \   \Omega_k(s, z)=\int_0^{z} t^{s-1} e^{P_0(t)} \, dt
$$
where $P_0(t)\in \CC[t]$ and $\omega_k$ is a root of unity pointing to a direction where the polynomial $P_0$ diverges to $-\infty$.
Some critical computations in the proof of the main Theorem generalize computations 
carried out for exponential periods appearing in \cite{BPM3}. 
The generalization of the Ramificant Determinant 
is the key result for the  proof of the linear 
independence of the Omega functions $(\Omega_k)$. In this magical calculation, we  compute 
a determinant of a matrix of exponential periods which are individually not computable.  
Omega functions appeared before in the literature
under the name of ``modified Gamma functions'' and their asymptotic behavior at infinite
was studied by N. G. De Bruijn, \cite{Bru} p.119, and A. Duval \cite{Du}. A. Aycock explained to me that
he also derived Omega functions from an old method by Euler to solve this type of functional 
equations (see \cite{[E123]} and \cite{Ayc3}). We know of no other earlier references
for Omega functions.

\section{Definition.}

Let $P_0(t)\in \CC[t]$ be a degree $d\geq 1$ polynomial such that $P_0(0)=0$ and 
$\lim_{t \to +\infty} \Re P_0(t) =-\infty$,  normalized such that
$$
P_0(t) = -\frac{1}{d} t^d +\sum_{k=1}^{d-1} a_k t^k 
$$
We also denote $a_d=-1/d$ and $a_0=0$. Let $\omega$ be the primitive $d$-th root of unity given by
$\omega =e^{\frac{2\pi i}{d}} $ and write $\omega_k=\omega^k$.

\begin{definition}
Let $d\geq 1$. For $k=0,1,\ldots , d-1$, the Omega functions, or $\Omega$-functions, associated 
to $P_0$, are defined for $\Re s >0$, by
$$
\Omega_k(s) =\int_0^{+\infty . \omega_k} t^{s-1} e^{P_0(t)} \, dt 
$$ 
\end{definition}
The roots $\omega_k$ point to the directions where the polynomial $P_0$ diverges exponentially to $-\infty$,
$$
\lim_{t\to +\infty . \omega_k} \Re P_0( t) =-\infty 
$$
so the integral is converging and we have a sound definition. There are different branches of $t^s=e^{s\log t}$ 
that we can take in the integral. The 
value of the integral will differ by a power of $e^{2\pi i s}$. For $k=0, \ldots d-1$, we choose the branch of $\log t$ with argument in 
$[0, 2\pi[$, With this convention we have for $u>0$, $(\omega_k u)^s =\omega_k^s u^s$ and this is important for the estimates. 

Usually, we spare the reference to $P_0$, but for some results it will be crucial 
to keep track of the dependence on parameters and we will write
$$
\Omega(s)=\Omega(s|P_0)=\Omega(s|a_1,\ldots , a_{d-1}) \ .
$$
For $d=1$, we have $P_0(t)=-t$ and $\Omega_1=\Gamma$ is Euler Gamma function.

If $P_0\in \RR[t]$, then $\Omega_0$ is real analytic.

Sometimes we will be interested in the case where $(a_k)$ are in a number field $\mathbb K\subset \CC$. 
In this case we say that these Omega functions are defined over $\mathbb K$.

\bigskip
\bigskip

\section{Meromorphic extension, poles and residues.}

\begin{theorem}\label{thm:functional_eq}
The Omega functions $(\Omega_k)_{0\leq k\leq d-1}$ extend to the complex plane into meromorphic functions 
of order $1$ satisfying the fundamental functional equation
\begin{equation}\label{eq:functional_eq}
\Omega_k(s+d) +\alpha_{d-1}\Omega_k(s+d-1) + \ldots + \alpha_1 \Omega_k(s+1) = s\, \Omega_k(s) 
\end{equation}
where $\alpha_l =-la_l$.

Moreover, the function $\Omega_k$ is holomorphic in $\CC-\NN$, and has simple poles at the negative integers. The  residue at 
$s=0$ is
$$
\Res_{s=0} \, \Omega_k =1 \ .
$$
\end{theorem}

Observe that for $d=1$, $\Omega_0=\Gamma$  and  
the functional equation (\ref{eq:functional_eq}) is  the classical functional 
equation $\Gamma (s+1)=s\Gamma (s)$. The general case of the functional equation when $\alpha_d\not= 1, 0$ 
is treated in the same way without the normalization $a_d=-1/d$.

\begin{proof}[\textbf{Proof.}]
For $\Re s >0$, we have by integration by parts,
\begin{align*}
 \Omega_k(s+d) +\sum_{l=1}^{d-1}\alpha_{l}\Omega_k(s+l) 
 &= \int_0^{+\infty . \omega_k} t^s(-P'_0(t)).e^{P_0(t)}\, dt\\ 
 &= \left [ -t^se^{P_0(t)}\right ]_0^{+\infty . \omega_k} +s \int_0^{+\infty . \omega_k} t^{s-1} e^{P_0(t)}\, dt\\
 &= s\, \Omega_k(s)
\end{align*}
and we get the functional equation (\ref{eq:functional_eq}). Now, using once the functional 
equation, we extend meromorphically $\Omega_k$ to $\{\Re s >-1\}$,
and by induction to $\{\Re s >-n\}$, for $n=1,2,\ldots$, hence to all of $\CC$. 
The only poles that can be introduced by this 
extension procedure using the functional equation are those created from the pole 
at $s=0$ and are at the negative integers. 
The functional equation
shows that $s \Omega_k(s)$ is holomorphic at $s=0$, hence the pole at $s=0$ is simple. It follows from the functional equation 
and the extension procedure that the other poles are also simple.
We compute the residue at $s=0$ using the functional equation,
$$
\Res_{s=0} \,  \Omega_k = \lim_{s\to 0} s \Omega_k(s)=\sum_{l=1}^d \alpha_l \Omega_k(l) = 
\int_0^{+\infty . \omega_k} (-P'_0(t)).e^{P_0(t)}\, dt = \left [ -e^{P_0(t)}\right ]_0^{+\infty . \omega_k}
=1
$$
\end{proof}
More generally, we can compute the residues at all the negative integers.

\begin{theorem}\label{thm:residues}
 Let  $(\lambda_n)_{n\geq 0}$ be the coefficients of the power series expansion of $e^{P_0(t)}$,
 $$
 e^{P_0(t)} = \sum_{n=0}^{+\infty } \lambda_n t^n \ .
 $$
 Then the residue of $\Omega_k$ at $s=-n$ is $\lambda_n$,
 $$
 \Res_{s=-n} \,  \Omega_k  = \lambda_n \ .
 $$
\end{theorem}

\begin{proof}[\textbf{Proof.}]
For $n\geq 0$ let $r_n\in \CC$ be the residue of $\Omega_k$ at $s=-n$, with $r_n=0$ if there is 
no pole, and $r_n=0$ for $n<0$. The functional equation (\ref{eq:functional_eq}) gives 
\begin{equation*}
r_n =\lim_{h\to 0} h\Omega_k (h-n) =\lim_{h\to 0}\frac{1}{h-n} \sum_{l=1}^d \alpha_l h\Omega_k(h-n+l) 
=- \frac{1}{n} \sum_{l=1}^d \alpha_l \, r_{n-l}
\end{equation*}
hence the recurrence relation
\begin{equation} \label{eq:rec}
nr_n =- \sum_{l=1}^d \alpha_l \, r_{n-l}
\end{equation}
Now, consider the generating power series
$$
F(t)=\sum_{n=0}^{+\infty} r_n t^n
$$
The recurrence relation (\ref{eq:rec})  gives
\begin{align*}
F'(t) = \sum_{n=0}^{+\infty} nr_n t^{n-1} &= -\sum_{l=1}^d \alpha_l \sum_{n=0}^{+\infty} r_{n-l} t^{n-1}\\
&= -\sum_{l=1}^d \alpha_l t^{l-1} \sum_{n=l}^{+\infty} r_{n-l} t^{n-l} \\
&= -\left (\sum_{l=1}^d \alpha_l t^{l-1} \right ) F(t) \\
&= P'_0(t) F(t)
\end{align*}
Since we have $F(0)=r_0=1$ from Theorem \ref{thm:functional_eq}, we get 
$$
F(t)=e^{P_0(t)}
$$
thus $r_n=\lambda_n$ as claimed.
\end{proof}

\textbf{Example.}

For $d=1$, the generating power series is
$$
F(t)=e^{-t} =\sum_{n=0}^{+\infty} \frac{(-1)^n}{n!} \, t^n
$$
and we recover the classical result that 
$$
\Res_{s=-n}\,  \Gamma =\frac{(-1)^n}{n!} \ .
$$

\medskip

Note that the residues at the simple poles at the negative integers are the same for all the 
functions $\Omega_0, \ldots, \Omega_{d-1}$. Indeed, for $k\not= l$, we can check directly 
that $\Omega_k-\Omega_l$ is an entire function because of the 
convergence for all $s\in \CC$ of the integral
$$
\Omega_k(s)-\Omega_l(s) = \int_{+\infty . \omega_l}^{+\infty . \omega_k} t^{s-1} e^{P_0(t)} \, dt 
=\int_{\g_{lk}} t^{s-1} e^{P_0(t)} \, dt
$$
where the integral can be taken over any path $\g_{lk}$ asymptotic to $+\infty . \omega_l$ and 
$+\infty . \omega_k$ in the proper direction and in $\CC-\RR_+$ (so with $0$ winding number around $0$). 
This integral depends holomorphically on the parameter $s\in \CC$.

Observe that if the coefficients of $P_0$ belong to a number field $\mathbb K$, $P_0(t)\in \mathbb K[t]$, then 
the residues of $\Omega_k$ belong also to $\mathbb K$. 
Another arithmetical observation  is the following:
\begin{corollary}
We assume that the only non-zero coefficients of $P_0$ are for powers divisible by an integer $n_0\geq 2$, that is, 
if $a_k\not=0$ then $n_0|k$. 

Then, if $n_0$ does not divide $n$, we have $r_n=0$.
\end{corollary}

\begin{proof}[\textbf{Proof.}]
From the previous Theorem we have
$$
e^{P_0(t)}=\prod_{k=1}^d e^{a_kt^k}= \prod_{k=1}^d \left (\sum_{m\geq 0}\frac{a_k^m}{m!} t^{mk} \right )
$$
and when we expand the last product we get the result.
\end{proof}

We have a more precise result than just the computation of the residues. We 
can determine the Mittag-Leffler decomposition of $\Omega_k$.This is an analytic result that requires some estimates.
\begin{theorem}
The Omega function $\Omega_k$ has the Mittag-Leffler decomposition:
$$
\Omega_k(s)=\sum_{n=0}^{+\infty} \frac{\lambda_n}{s+n} +\int_1^{+\infty . \omega_k} t^{s-1}  e^{P_0(t)} \, dt 
$$
where the integral is an entire function of order $1$.
\end{theorem}

The path of integration of the integral without its endpoint at $1$ is contained in $\CC-\RR_+$.
Observe that this Theorem shows that the Omega function $\Omega_k$ is a meromorphic function 
of order $1$.

\begin{corollary}
The Omega functions $\Omega_k$ are meromorphic functions of order $1$. 
 
\end{corollary}

\begin{proof}[\textbf{Proof.}]
We write
$$
\Omega_k(s)= \int_0^1 t^{s-1}  e^{P_0(t)} \, dt +\int_1^{+\infty.\omega_k} t^{s-1}  e^{P_0(t)} \, dt 
$$
and we compute the first integral expanding the exponential in power series (uniformly convergent in $[0,1]$).
For $\Re s >0$, we have 
\begin{equation*}
\int_0^{1} t^{s-1}  e^{P_0(t)} \, dt =\sum_{n=0}^{+\infty} \lambda_n \int_0^{1} t^{s+n-1} \, dt 
=\sum_{n=0}^{+\infty}  \lambda_n\left[\frac{t^{s+n}}{s+n}\right]_0^1 =\sum_{n=0}^{+\infty}  \frac{\lambda_n}{s+n} \ .
\end{equation*}
The second integral can be bounded by the next Lemma that shows that it is an entire function of order $1$ 
(using that Euler Gamma function is of order $1$).
\end{proof}

\begin{lemma}
For universal constants $C_0, C_1 >0$, and $s\in \CC$, we have the estimate
$$
\left | \int_1^{+\infty . \omega_k} t^{s-1}  e^{P_0(t)} \, dt \right | \leq e^{-2\pi  \frac{k}{d} \Im s}  \left (C_0 + C_1 
d^{\Re s /d} \Gamma\left (\frac{\Re s}{d}\right )  \right )
$$
\end{lemma}

\begin{proof}[\textbf{Proof.}]
We make the change of variables $t=\omega_k u$
\begin{align*}
\int_1^{+\infty . \omega_k} t^{s-1} e^{P_0(t)} \, dt &= 
\omega_k^s \int_{\omega_k^{-1}}^{+\infty} u^{s-1} e^{-\frac{1}{d} u^d (1+\cO (u^{-1}))} \, du \\
&=e^{2\pi i \frac{k}{d} s} \int_{\omega_k^{-1}}^{+\infty} u^{s-1} e^{-\frac{1}{d} u^d (1+\cO (u^{-1}))} \, du
\end{align*}
This gives the bound
\begin{equation*}
\left |  \int_1^{+\infty . \omega_k} t^{s-1} e^{P_0(t)} \, dt \right | \leq   e^{-2\pi  \frac{k}{d} \Im s} 
\left | \int_{\omega_k^{-1}}^{+\infty} u^{s-1} e^{-\frac{1}{d} u^d (1+\cO (u^{-1}))} \, du \right |
\end{equation*}
Now, taking an integration path of uniform finite length from $1$ to $\omega_k$, bounded away from $0$, and not crossing $\RR_+$ 
(hence not winding around $0$), we get (using  
$C$ to denote several universal constants $C>0$ depending only on $d$ and $P_0$)
\begin{equation*}
\left |  \int_1^{+\infty . \omega_k} t^{s-1} e^{P_0(t)} \, dt \right | \leq   e^{-2\pi  \frac{k}{d} \Im s} \left ( C+
\left | \int_{1}^{+\infty} u^{s-1} e^{-\frac{1}{d} u^d (1+\cO (u^{-1}))} \, du \right | \right )
\end{equation*}
The last integral can be estimated by
\begin{align*}
 \left | \int_{1}^{+\infty} u^{s-1} e^{-\frac{1}{d} u^d (1+\cO (u^{-1}))} \, du \right | &\leq (1+C)\int_1^{+\infty} u^{\Re s-1} 
 e^{-\frac{1}{d} u^d } \, du \leq (1+C) \int_0^{+\infty} u^{\Re s-1}  e^{-\frac{1}{d} u^d } \, du \\
 &\leq (1+C) \, d^{\Re s/d} \, \Gamma\left (\frac{\Re s}{d}\right ) 
\end{align*}
(for the computation of the last integral we use the change of variable $v=u^d/d$).
\end{proof}

\section{Incomplete Omega functions.}

We define the Incomplete Omega functions that generalize the Incomplete Gamma function.

\begin{definition}
For $z, s\in \CC$, $\Re s >0$,  the Incomplete Omega function $\Omega(s,z)$ is defined by
\begin{equation*}
\Omega(s,z) =\int_0^{z} t^{s-1} e^{P_0(t)} \, dt 
\end{equation*}
\end{definition}
The integration path from $0$ to $z$ is choosen not crossing $\RR_+$ (hence not winding around $0$). 
For $P_0(t)=-t$ this is the classical Incomplete Gamma function.
Observe that we recover all the $(\Omega_k)$ functions by taking the appropriate limit of $\Omega(s,z)$ as $z\to \infty$, 
$$
\Omega_k(s)=\lim_{z\to +\infty.\omega_k} \Omega (s,z)
$$
For the particular integral values $s=1,2,\ldots, d-1$ these Incomplete Omega functions 
are the transcendental entire functions of the variable $z\in \CC$
studied in \cite{BPM3} which form a basis of the fundamental 
vector space of functions on a simply connected log-Riemann surface
with exactly $d$ infinite ramification points. Some of the results in this section generalize some results from \cite{BPM3}.
Following the same Abel's philosophy that inspires \cite{BPM3}, we prove that we only need to use a finite number 
of transcendentals $(\Omega(s+k,z))_{0\leq k\leq d-1}$ to compute integrals of the form 
$$
\int_0^{z} t^{s-1}Q(t) e^{P_0(t)} \, dt
$$
where $Q$ is a polynomial.

\begin{proposition}
Let $Q(t)\in \CC[t]$. For $d\geq 2$, the integral 
$$
\int_0^{z} t^{s-1} Q(t) e^{P_0(t)} \, dt
$$
is of the form
\begin{equation*}
\int_0^{z} t^{s-1} Q(t) e^{P_0(t)} \, dt = z^{s} A(s,z) \, e^{P_0(z)} + \sum_{k=0}^{d-1} c_k(s) \, \Omega(s+k, z) \,
\end{equation*}
where $A \in \CC[s,z]$ is a polynomial,  and the polynomial coefficients $c_k(s)\in \CC[s]$ do have coefficients
depending polynomially on the coefficients $(a_1, \ldots , a_{d-1})$ of $P_0$.
\end{proposition}

\begin{proof}[\textbf{Proof.}]
Observe that we have
\begin{equation*}
 \int_0^{z} t^{s-1} Q(t) e^{P_0(t)} \, dt = Q(0) \Omega(s,z) + \int_0^{z} t^{s-1} (Q(t)-Q(0)) e^{P_0(t)} \, dt
\end{equation*}
and we can assume that $Q(0)=0$. We are reduce to prove the result for $Q(t) = t R(t)$, or to
consider the integral
\begin{equation*}
 \int_0^{z} t^{s-1} Q(t) e^{P_0(t)} \, dt =  \int_0^{z} t^{s} R(t) e^{P_0(t)} \, dt
\end{equation*}

First we consider the case $d=1$.  We prove the result for  $R(t)=t^n$ integrating by parts $n+1$ times
\begin{align*}
&\int_0^z t^{s+n} e^{-t} \, dt =\left [-t^{s+n} e^{-t}\right ]_0^z +(s+n) \int_0^z t^{s+n-1} e^{-t} \, dt \\
&=-z^{s+n} e^{-z} +(s+n) \int_0^z t^{s+n-1} e^{-t} \, dt \\
&=-(z^{s+n} +(s+n)z^{s+n-1})e^{-z}+ (s+n)(s+n-1)\int_0^z t^{s+n-2} e^{-t} \, dt \\
&\vdots \\
&=-z^s(z^{n} +(s+n)z^{n-1}+\ldots )e^{-z} + (s+n)(s+n-1)\ldots s \, \Omega(s,z)
\end{align*}
For a general polynomial $R(t)$ we have the result by linear decomposition of the integral.

In the rest of the proof we assume $d\geq 2$.
If $q=\deg R \leq d-2$, by linearity, the integral is a linear combination
(with the coefficients of $tR(t)$) of $(\Omega(s+k,z))_{0\leq k\leq d-1}$ and the result follows.

If $\deg R \geq d-1$, then we consider the Euclidean division of $R(t)$ by $P_0'(t)$,
$$
R(t)=A_1(t)P_0'(t)+B_1(t)
$$
with $A_1, B_1\in \CC[t]$, $\deg B_1 \leq  d-2$ and $\deg A_1 =\deg R-(d-1) =q-(d-1)\leq q-1$.
We proceed  splitting the integral:
\begin{equation*}
\int_0^{z} t^{s} R(t) e^{P_0(t)} \, dt = \int_0^z t^{s} A_1(t)P_0'(t) e^{P_0(t)} \, dt + \int_0^z t^{s} B_1(t) e^{P_0(t)} \, dt
\end{equation*}
Since $\deg B_1 \leq d-2$, the second integral is a linear combination 
of $\Omega(s,z), \Omega(s+1,z),\ldots ,\Omega(s+d-1,z)$,
thus of the desired 
form, and we can forget about it. We work on the first integral integrating by parts, 
\begin{equation*}
\int_0^z t^{s} A_1(t)P_0'(t) e^{P_0(t)} \, dt = 
\left [t^{s} A_1(t)e^{P_0(t)}\right ]_0^z-\int_0^z \left (t^{s} A_1(t)\right )'  e^{P_0(t)} \, dt 
\end{equation*}
Then we get:
\begin{equation*}
\int_0^z t^{s} A_1(t)P_0'(t) e^{P_0(t)} \, dt = z^{s} A_1(z)e^{P_0(z)}- \int_0^z t^{s} A'_1(t)  e^{P_0(t)} \, dt 
-  s \int_0^z t^{s-1} A_1(t)  e^{P_0(t)} \, dt
\end{equation*}
The first integral in the right hand side is of the same form as the initial one with $R(t)$ but
with $\deg A'_1 =\deg A_1-1\leq \deg R -(d-1)-1= q-d\leq q-2$ (using here $d\geq 2$), hence by descending
induction we can forget about it. For the second integral, we can write
$$
s\int_0^z t^{s-1} A_1(t)  e^{P_0(t)} \, dt =A_1(0) s \, \Omega(s,z) + s 
\int_0^z t^{s} \left (\frac{A_1(t)-A_1(0)}{t} \right )  e^{P_0(t)} \, dt
$$
and $t^{-1} (A_1(t)-A_1(0))$ is a polynomial of degree $\deg A_1 -1\leq q-2$. 
Then the descending induction gives the expression for the integral as announced.

The coefficients of $P_0$ appear first linearly in the Euclidean divisions by $P_0'$ then, by repeated 
Euclidean divisions the dependence of the $c_k(s)$ is polynomial on the coefficients of $P_0$.
\end{proof}

Now, we can prove that the Omega functions $(\Omega_k(s))_{0\leq k\leq d-1}$ generate a large class of exponential periods:

\begin{corollary} \label{cor:computation}
Let $Q(t)\in \CC[t]$ and $0\leq n\leq d-1$. The exponential period
$$
\int_0^{+\infty.\omega_n} t^{s-1}Q(t) \, e^{P_0(t)} \, dt
$$
is a linear combination of the exponential periods $(\Omega_k(s))_{0\leq k\leq d-1}$
$$
\int_0^{+\infty.\omega_n} t^{s-1}Q(t) \, e^{P_0(t)} \, dt =\sum_{k=0}^{d-1} c_k(s) \, \Omega_n(s+k)
$$
where the coefficients $c_k(s)$ are polynomials on $s$ and on the coefficients $(a_1,\ldots , a_{d-1})$.
\end{corollary}

\begin{proof}[\textbf{Proof.}]
For $d\geq 2$, we have from the previous Proposition that 
\begin{equation*}
\int_0^{z} t^{s-1}Q(t) e^{P_0(t)} \, dt = A(s, z,z^s)e^{P_0(z)} + \sum_{k=0}^{d-1} c_k(s)\Omega(s+k, z)
\end{equation*} 
When $z\to +\infty.\omega_n$ the first term in the right side vanish, since 
$A(s,z,z^s)e^{P_0(z)} \to 0$ for a polynomial $A(s,z,z^s)\in \CC[z]$
(the exponential decay of $e^{P_0(z)}$ takes over the polynomial divergence of $A(z)$), 
and we get
\begin{equation*}
\int_0^{+\infty.\omega_n} t^{s-1}Q(t) e^{P_0(t)} \, dt = \sum_{k=0}^{d-1} c_k(s)\Omega_n(s+k) 
\end{equation*}

\end{proof}

\section{Linear independence.}

The row vector build with  Omega functions $\mathbf{\Omega}(s)=(\Omega_k(s))_{0\leq k\leq d-1}$ has the following 
important linear independence property:

\begin{theorem}\label{th:linear_independence}
For any $s\in \CC-\NN_-^*$, the vectors  $\mathbf{\Omega}(s+1), \mathbf{\Omega}(s+2), \ldots , 
\mathbf{\Omega}(s+d)$ are linearly independent, 
$$
\Delta (s) \not= 0
$$
where 
$$
\Delta(s| a_1, \ldots , a_{d-1})=\det  \left [ 
\begin{matrix}
\Omega_{11} & \Omega_{12} & \ldots &\Omega_{1d} \\
\Omega_{21} & \Omega_{22} & \ldots &\Omega_{2d} \\ 
\vdots &\vdots &\ddots & \vdots \\
\Omega_{d1} & \Omega_{d2} & \ldots &\Omega_{dd} \\ 
\end{matrix}
\right ]  \ \ .
$$
where $\Omega_{kl} =\Omega_{k-1}(s+l)$. 

More precisely, we can compute 
$$
\Delta (s| a_1, \ldots , a_{d-1}) =\Delta (s| 0, \ldots , 0) \exp \left ( \Pi_d(s, a_1, \ldots, a_{d-1}) \right )
$$
where $\Pi_d(s, a_1, \ldots, a_{d-1})$ is a universal polynomial with rational coefficients.
\end{theorem}

In view of the last formula, the result follows from  $\Delta (s| 0, \ldots , 0) \not= 0$. We will prove the last 
formula and compute explicitly the determinant $\Delta (s| 0, \ldots , 0)$. 
These computations are similar to the ones for 
the Ramificant Determinant (see \cite{BPM3}) that corresponds to the special case $s=0$.

We can compute $\Omega_k(s+l|0, \ldots , 0) $ using Euler Gamma function.

\begin{lemma}\label{lemma:computation}
We have
$$
\Omega_k(s+l|0, \ldots , 0) 
= \omega^{k(s+l)} d^{\frac{s+l}{d}-1} \Gamma\left ( \frac{s+l}{d}\right )
$$
\end{lemma}

\begin{proof}[\textbf{Proof.}]
We first make the change of variables $t=\omega^k u$, and then $v=u^d/d$,
\begin{align*}
 \Omega_k(s+l|0, \ldots , 0) &= \int_0^{+\infty . \omega^k} t^{s+l-1} e^{-\frac{1}{d} t^d} \, dt \\
 &= \omega^{k(s+l)} \int_0^{+\infty } u^{s+l-1} e^{-\frac{1}{d} u^d} \, du \\
 &= \omega^{k(s+l)} d^{\frac{s+l}{d}-1} \int_0^{+\infty } v^{\frac{s+l}{d}-1} e^{-v} \, dv \\
 &= \omega^{k(s+l)} d^{\frac{s+l}{d}-1} \Gamma\left ( \frac{s+l}{d}\right ) 
\end{align*}

\end{proof}

Now we recall the following well known elementary Vandermonde Lemma:

\begin{lemma}\label{lemma:Vandermonde}
If $\xi_1 ,\ldots , \xi_d$ are the $d$ roots
of a monic polynomial $Q(X)$, then we can compute the 
Vandermonde determinant $V(\xi_1 ,\ldots , \xi_d)$ of the $(\xi_1 ,\ldots , \xi_d)$ as
\begin{equation*}
V(\xi_1 ,\ldots , \xi_d)= \left | 
\begin{matrix}
1 &\xi_1 &\xi_1^2 &\ldots &\xi_1^{d-1} \\ 
1 &\xi_2 &\xi_2^2 &\ldots &\xi_2^{d-1} \\ 
\vdots &\vdots&\vdots &\ddots &\vdots \\ 
1 &\xi_d &\xi_d^2 &\ldots &\xi_d^{d-1} \\ 
\end{matrix}
\right |
=\prod_{i\not= j} (\xi_i-\xi_j)=\prod_{i=1}^d Q'(\xi_i ) \ .
\end{equation*}
\end{lemma}

Using this Lemma with $Q(X)=X^d-1$ we compute the Vandermonde determinant:
\begin{align*}
V_d&=\left | 
\begin{matrix}
1 &\omega_1 &\omega_1^2 &\ldots &\omega_1^{d-1} \\
1 &\omega_2 &\omega_2^2 &\ldots &\omega_2^{d-1} \\ 
\vdots &\vdots&\vdots &\ddots &\vdots \\ 
1 &\omega_d &\omega_d^2 &\ldots &\omega_d^{d-1} \\ 
\end{matrix}
\right | 
=\prod_{i\not= j} (\omega_i-\omega_j) =\prod_i (d\omega_i^{d-1})=\\
&=d^d \left ( \prod_i \omega_i \right
)^{d-1}=(-1)^{(d-1)^2} d^d = (-1)^{d-1} d^d\ .
\end{align*}
We use this result to compute $\Delta (s| 0, \ldots , 0)$.

\begin{lemma}\label{lemma:comp}
We have
$$
\Delta (s| 0, \ldots , 0) = 
\frac{\left ( 2 \pi d \right )^{\frac{d}{2}}}{\sqrt{2\pi}} \omega^{\frac{d(d-1)}{2}s} \,  \Gamma (s+1)
$$
and in particular $\Delta (s| 0, \ldots , 0) \not= 0$ for $s\not= -1, -2,  \ldots$.
\end{lemma}

Taking the value $s= 0$ we recover the formula from Lemma 3.5 from \cite{BPM3}.

\begin{proof}[\textbf{Proof of Lemma \ref{lemma:comp}.}]
Using Lemma \ref{lemma:computation} we have
\begin{align*}
&\Delta (s|0,\ldots , 0) =\\
&=\left | 
\begin{matrix}
\omega^{0.(s+1)} d^{\frac{s+1}{d}-1} \Gamma \left (\frac{s+1}{d}\right )  & 
\omega^{0.(s+2)}   d^{\frac{s+2}{d}-1} \Gamma \left (\frac{s+2}{d}\right ) & 
\ldots &
\omega^{0.(s+d)}  d^{\frac{s+d}{d}-1} \Gamma \left (\frac{s+d}{d}\right ) \\
\omega^{1.(s+1)} d^{\frac{s+1}{d}-1} \Gamma \left (\frac{s+1}{d}\right )  &
\omega^{1.(s+2)} d^{\frac{s+2}{d}-1} \Gamma \left (\frac{s+2}{d}\right ) & 
\ldots &
\omega^{1.(s+d)} d^{\frac{s+d}{d}-1} \Gamma \left (\frac{s+d}{d}\right )  \\ 
\vdots &\vdots &\ddots & \vdots \\ 
\omega^{(d-1).(s+1)} d^{\frac{s+1}{d}-1} \Gamma \left (\frac{s+1}{d}\right )  & 
\omega^{(d-1).(s+2)}d^{\frac{s+2}{d}-1} \Gamma \left (\frac{s+2}{d}\right )  & 
\ldots &
\omega^{(d-1).(s+d)}d^{\frac{s+d}{d}-1} \Gamma \left (\frac{s+d}{d}\right )  \\  
\end{matrix}
\right | \\
&=\omega^{\frac{d(d-1)}{2}s}  d^{s} d^{\frac{d+1}{2} -d} 
\Gamma \left (\frac{s}{d}+\frac{1}{d}\right ) \ldots
\Gamma \left (\frac{s}{d}+\frac{d-1}{d}\right ) 
\Gamma \left (\frac{s}{d}+\frac{d}{d}\right ) 
\left | 
\begin{matrix}
\omega_0^1 & \omega_0^2 & \ldots & \omega_0^d\\ 
\omega_1^1 & \omega_1^2 & \ldots & \omega_1^d \\
\vdots &\vdots &\ddots & \vdots \\
\omega_{d-1}^1 & \omega_{d-1}^2 & \ldots & \omega_{d-1}^d \\ 
\end{matrix}
\right | \\
&=\omega^{\frac{d(d-1)}{2}s}  d^{s} d^{\frac{d+1}{2} -d} 
\Gamma \left (\frac{s}{d}+\frac{1}{d}\right ) \ldots
\Gamma \left (\frac{s}{d}+\frac{d-1}{d}\right )
\frac{s}{d}\, \Gamma \left (\frac{s}{d}\right ) 
\, d^d \\
&=\omega^{\frac{d(d-1)}{2}s}  d^{s} d^{\frac{d+1}{2} -d} 
 \left ( 2\pi  \right )^{\frac{d-1}{2}} d^{\frac{1}{2}-s} \Gamma(s)
\frac{s}{d}\, 
\, d^d \\
&=\omega^{\frac{d(d-1)}{2}s}  d^{\frac{d}{2}} \, (2\pi)^{\frac{d-1}{2}} 
\, s \Gamma (s) \\
&=d^{\frac{d}{2}} \left ( 2\pi  \right )^{\frac{d-1}{2}} \omega^{\frac{d(d-1)}{2}s} \, \Gamma (s+1)
 \\
&=\frac{\left ( 2 \pi d \right )^{\frac{d}{2}}}{\sqrt{2\pi}} \omega^{\frac{d(d-1)}{2}s} \, \Gamma (s+1)
 \\
\end{align*}
where we have used Gauss multiplication formula (that is in fact due 
to Euler and not to Gauss, see \cite{Ayc2}) with $z=s/d$,
$$
\Gamma (z ).\Gamma \left ( z+\frac{1}{d}\right )\ldots 
\Gamma \left (z+\frac{d-1}{d}\right )=(2\pi)^{\frac{d-1}{2}} d^{\frac{1}{2}-dz}
\Gamma(dz) 
$$
so
$$
\Gamma \left (\frac{s}{d}\right ) \Gamma \left (\frac{s}{d}+\frac{1}{d}\right ) \ldots\Gamma \left (\frac{s}{d}+\frac{d-1}{d}\right ) = (2\pi)^{\frac{d-1}{2}} d^{\frac{1}{2}-s} \Gamma(s) \ .
$$
The determinant in the fourth line of the computation is equal to
\begin{align*}
&\left | 
\begin{matrix}
\omega_0^1 & \omega_0^2 & \ldots & \omega_0^d\\ 
\omega_1^1 & \omega_1^2 & \ldots & \omega_1^d \\
\vdots &\vdots &\ddots & \vdots \\
\omega_{d-1}^1 & \omega_{d-1}^2 & \ldots & \omega_{d-1}^d \\ 
\end{matrix}
\right |
=
\left | 
\begin{matrix}
\omega_0^1 & \omega_0^2 & \ldots & 1 \\ 
\omega_1^1 & \omega_1^2 & \ldots & 1 \\
\vdots &\vdots &\ddots & \vdots \\
\omega_{d-1}^1 & \omega_{d-1}^2 & \ldots & 1 \\ 
\end{matrix}
\right | 
=  \left | 
\begin{matrix}
1&\omega_0  & \ldots & \omega_0^{d-1} \\ 
1&\omega_1 & \ldots & \omega_1^{d-1} \\
\vdots &\vdots &\ddots & \vdots \\
1&\omega_{d-1} & \ldots & \omega_{d-1}^{d-1} \\ 
\end{matrix}
\right | =\\
&=(-1)^{d-1}V_d =d^d
\end{align*}
where $V_d$ is the Vandermonde determinant computed previously.
\end{proof}

\begin{proof}[\textbf{Proof of Theorem \ref{th:linear_independence}.}]
Consider the entire function of several complex variables $\Delta (s|a_1,a_2,\ldots , a_{d-1})$
on the variables $(a_1,a_2,\ldots , a_{d-1})$.
Observe that Corollary \ref{cor:computation} proves that
each integral
$$
\int_0^{+\infty . \omega_k} t^{s+n-1} e^{P_0(t)} \ dt \ ,
$$
is a linear combination with coefficients that are polynomial 
on $s$ and the $(a_j)$ of the integrals $\Omega_k(s)$ for $k=0,1,\ldots,d-1$,
Therefore, differentiating column by column, we observe that each differentiation of a column,
gives a linear combination of all the columns, and they all give a vanishing determinant except
for the same column.  Thus, for
each $k=0,1,\ldots , d-1$, we have
$$
\partial_{a_k} \Delta = Q_k \, \Delta 
$$
where $Q_k$ is a polynomial on $s$ and the $(a_j)$.
We conclude that the logarithmic derivative of
$\Delta$ with respect to each variable $a_k$ is a universal polynomial
on the variables $s$ and $(a_j)$. This gives the
existence of the universal polynomial $\Upsilon_d$ such that
$$
\Delta (s|a_1,a_2,\ldots ,a_{d-1})=c(s).e^{\Upsilon_d(s; a_1,a_2,\ldots ,a_{d-1})} \ ,
$$
with  $c(s).e^{\Upsilon_d(s; 0,0,\ldots ,0)}=\Delta (s|0,\ldots , 0) \in \CC$. Then 
if we define $\Pi_d(s,a_1,\ldots, a_{d-1})=\Upsilon_d(s; a_1,a_2,\ldots ,a_{d-1})-\Upsilon_d(s; 0,0,\ldots ,0)$
we get the result
$$
\Delta (s|a_1,a_2,\ldots ,a_{d-1}) = \Delta (s|0,\ldots , 0) e^{\Pi_d(s,a_1,\ldots, a_{d-1})}
$$
\end{proof}

\begin{corollary}
The functions $\Omega_0,\ldots , \Omega_{d-1}$ do not have a common zero in $\CC-\NN_-$. 
\end{corollary}

\begin{proof}[\textbf{Proof.}]
Otherwise, if $s_0\in \CC-\NN_-$ is a common zero, then $s_0+1\in \CC-\NN_-^*$ and 
the functional equation shows that the non-zero 
vector $(1, \a_{d-1},\ldots, \alpha_1)$ is 
in the kernel of the matrix $[\Omega_{kl}(s_0+1)]$, which contradicts that it has
non-vanishing
determinant by Theorem \ref{th:linear_independence}.
\end{proof}

Observe that this simultaneous non-vanishing result relies on the fact that Euler Gamma function
has no zeroes. This is something that was explained to be a ``mini-Riemann hypothesis'' in \cite{PM1}, and was the
subject of correspondence between Hermite and Stieltjes \cite{HS}. Although used in the proof, 
the non-vanishing of Euler Gamma 
function is a particular case of this general result for Omega functions.

\begin{corollary}
The functions $\Omega_0,\ldots , \Omega_{d-1}$ are $\CC$-linearly independent. 
\end{corollary}

\begin{proof}[\textbf{Proof.}]
Otherwise there will be a non-trivial null linear combination of the 
rows of the matrix  $[\Omega_{kl}]$ and the determinant 
will be zero.
\end{proof}

\section{Solutions of the functional equation.}

Observe that the functional equation (\ref{eq:functional1}) reduces to the functional 
equation (\ref{eq:functional_eq}) by dividing the equation by $\a_d$ that is assumed to be non-zero.
We can make a first observation that the space of solutions of the 
functional equation (\ref{eq:functional_eq}) is an infinite dimensional vector space.
 
\begin{proposition}
The space of meromorphic solutions $f$ of the functional equation
\begin{equation}
f(s+d) +\alpha_{d-1}f(s+d-1) + \ldots + \alpha_1 f(s+1) = s\, f(s) 
\end{equation}
is an infinite dimensional vector space.
\end{proposition}

\begin{proof}[\textbf{Proof.}]
The functional equation is linear and there are non-zero solutions (the $\Omega_k$ functions). 
Given a non-zero meromorphic solution $f(s)$, 
we can construct an infinite number of linearly independent solutions
$$
g(s)=e^{2\pi i ns} f(s)
$$
where $n\in \ZZ$ is any integer.
\end{proof}

If we restrict to solutions with a controlled growth, then we prove that the space of solutions is finite dimensional.

\begin{definition} \label{def:exp_growth}
 We consider the $\CC$-vector space $\mathbb{V}$ of meromorphic functions $f$ satisfying the 
functional equation (\ref{eq:functional_eq}) and the estimate in the vertical 
strip  $S_{1,d} =\{1\leq \Re s \leq d\}$, for $s \in S_{1,d}$, 
\begin{equation}\label{eq:estimate}
\left | f (s) \right | \leq  C e^{-c  \Im s} 
\end{equation}
for some constants $C>0$ and $0\leq c<2\pi$. 
\end{definition} 

It is clear that the space $\mathbb{V}$ is a subspace of the vector space of general solutions 
(without a prescribed growth condition).
We prove first that $\mathbb{V}$ is non-empty by proving the estimates for the functions $\Omega_k$
for $k=0,1,\ldots d-1$.

\begin{proposition}
For $k=0,1,\ldots, d-1$, for any strip $S_{a,b} =\{a\leq \Re s \leq b\}$ with $0<a<b$, there exists 
a constant $C=C(a, b,P_0)>0$, depending only on $a, b>0$ and the polynomial $P_0$, such that for $s \in S_{a,b}$, 
we have
$$
\left | \Omega_k(s) \right | \leq C e^{-\frac{2\pi  k}{d} \Im s}
$$
\end{proposition}

Obviously we can take $a=1$ and $b=d$ and since $0\leq c=\frac{2\pi  k}{d} <2\pi$ we get 
that $\Omega_k$ satisfies the estimate (\ref{eq:estimate}).
\begin{proof}[\textbf{Proof.}]
We make the change of variables $t=\omega_k u$
\begin{align*}
\Omega_k(s) &= \int_0^{+\infty . \omega_k} t^{s-1} e^{P_0(t)} \, dt \\
&= \omega_k^s \int_0^{+\infty} u^{s-1} e^{-\frac{1}{d} u^d (1+\cO (u^{-1}))} \, du \\
&=e^{2\pi i \frac{k}{d} s} \int_0^{+\infty} u^{s-1} e^{-\frac{1}{d} u^d (1+\cO (u^{-1}))} \, du
\end{align*}
so, we get for $0<a\leq \Re s\leq b$
\begin{align*}
|\Omega_k(s)| &\leq e^{-2\pi  \frac{k}{d} \Im s} (1+C_1) \int_0^{+\infty } 
u^{\Re s-1} e^{-\frac{1}{d} u^d} \, du \\
&\leq e^{-2\pi  \frac{k}{d} \Im s} (1+C_1) d^{\Re s /d} 
\Gamma\left (\frac{\Re s}{d}\right ) \\
&\leq C e^{-\frac{2\pi  k}{d} \Im s}
\end{align*}
where $C, C_1 >0$ are constants depending only on $a,b>0$ and $P_0$.
\end{proof}

The growth condition on the strip $S(1,d)$ and the functional equation implies a control of $f$ in the half plane
$\{\Re s \geq 1\}$. More precisely, we have the following Proposition:

\begin{proposition}\label{prop:exp_growth}
Let $f\in \mathbb{V}$. Then
there exists constants $C_0, \tau >0$ only depending on $d$ and the coefficients $\alpha_1, \ldots ,\alpha_{d-1}$ 
such that for $\Re s \geq 1$
$$
|f(s)|\leq C_0 e^{\tau \Re s} e^{\frac{\pi}{2} |\Im s|} \, \left | \Gamma\left ( \frac{ s} {d}  \right ) \right | e^{-c \Im s}
$$
where $0\leq c < 2 \pi$ is the constant from estimate (\ref{eq:estimate}) in Definition \ref{def:exp_growth}.
\end{proposition}

In the proof we use the following two classical estimates for Euler Gamma function:

\begin{lemma}\label{lem:gamma_estimate1}
For $s=\sigma+i\eta$, $\sigma \geq 1/2$, 
$$
\left |\Gamma (s)\right | \geq \frac{\Gamma(\sigma)}{(\cosh (\pi \eta)^{1/2}} \geq \Gamma(\sigma) e^{-\pi |\eta|/2}
$$
\end{lemma}

\begin{lemma}\label{lem:gamma_estimate2}
For $\Re s \geq 0$, $a\geq 0$ and $0<b-a<1$,
$$
\left |\frac{ \Gamma(s+a)}{ \Gamma(s+b)} \right | \leq \frac{1+b/|s|}{|s|^{b-a} }\ .
$$
\end{lemma}

For the first lower estimate in Lemma \ref{lem:gamma_estimate1} see \cite{Car} p.51, inequality 3.10-3.
For the second estimate in Lemma \ref{lem:gamma_estimate2} see \cite{PK}, p.34, inequality 2.1.17.

\begin{proof}[\textbf{Proof of Proposition \ref{prop:exp_growth}}] 
Consider the function $g(s) = f(s)/\Gamma( s/d )$
that is well defined for $\Re s>0$ since the Gamma function does not vanish in the right half plane.
We have
$$
g(s+d) = \frac{f(s+d)}{\Gamma\left ( \frac{s}{d} +1\right )} =\frac{d}{s} \frac{f(s+d)}{\Gamma\left ( \frac{s}{d} \right )} 
$$
and using the functional equation for $f$, 
$$
g(s+d)=d g(s)-\alpha_1 \frac{ \Gamma\left (\frac{s+1}{d}\right )}{ \Gamma\left (\frac{s+d}{d}\right )} g(s+1)+\ldots 
-\alpha_{d-1} \frac{ \Gamma\left (\frac{s+d-1}{d}\right )}{ \Gamma\left (\frac{s+d}{d}\right )} g(s+d-1) \  .
$$
For $\Re s \geq 1$ and $k=1, \ldots d-1$, the coefficients $\frac{ \Gamma\left (\frac{s+k}{d}\right )}{ \Gamma\left (\frac{s+d}{d}\right )}$ are bounded and go to zero when $\Re s\to +\infty$ as follows from Stirling asymptotics. We 
can be  more precise and use the 
estimate from Lemma \ref{lem:gamma_estimate2} with $s$ replaced by $s/d$, $b=1$ and $a=k/d$. 
We get, for $\Re s\geq 1$ (so $|s|\geq 1$),
$$
\left | \frac{ \Gamma\left (\frac{s+k}{d}\right )}{ \Gamma\left (\frac{s+d}{d}\right )} \right |\leq C_1
$$
where the constant $C_1>0$ only depends on $d$ (we can take $C_2=(1+d) d^{1/d-1}$).

It follows that for $\Re s\geq 1$ we have
$$
|g(s+d)| \leq C_2 (|g(s)| +|g(s+1)|+\ldots +|g(s+d-1)|)
$$
where the constant $C_2>0$ only depends on $d$ and $\alpha_1, \ldots ,\alpha_{d-1}$.

We consider now 
$$
M(\sigma, \eta)=\max_{1\leq \Re s\leq \sigma \atop \Im s=\eta} |g(s)|  \ .
$$

For $\sigma \geq d$, using the previous inequaity, we have that 
$$
M(\sigma+1, \eta) \leq C_3 M(\sigma, \eta) 
$$
where $C_3=d.C_2$. Hence, by induction, for a positive integer $k\geq 1$, and $d\leq \sigma \leq d
+k$, we have
$$
M(\sigma, \eta) \leq M(d+k, \eta)\leq C_3^k M(d, \eta)    
$$
therefore, for $\Re s\geq d$
$$
|g(s)|\leq M(\Re s, \Im s)\leq M([\Re s ]+1, \Im s )\leq C_3^{[\Re s ]+1-d} M(d,  \Im s )\leq  C_4. C_3^{\Re s}M(d,  \Im s )
$$
where $C_4=C_3^{1-d}$ depends only on $d$ and $\alpha_1, \ldots ,\alpha_{d-1}$.

 Now, from Definition \ref{def:exp_growth}, we have, for $\Re s \geq 1$,
$M(d, \Im s) \leq C. \left |\Gamma\left ( \frac{s}{d}\right )) \right |^{-1}. e^{-c\Im s} $ and using the lower 
estimate from Lemma \ref{lem:gamma_estimate1} we get
$$
M(d,  \Im s )\leq C . e^{-c\Im s}  e^{\frac{\pi}{2} |\Im s|} \ .
$$

Putting everything together, we have for  $\Re s \geq 1$,
$$
|g(s)|\leq C_5. C_4^{\Re s}e^{-c\Im s}  e^{\frac{\pi}{2} |\Im s|}
$$
and
$$
|f(s)|\leq C_5. C_4^{\Re s}  e^{\frac{\pi}{2} |\Im s|}  \, \left | \Gamma\left ( \frac{ s} {d}  \right ) \right |  e^{-c\Im s}  
$$
which proves the estimate of the Proposition.
\end{proof}

Now we can prove the main Theorem:

\begin{theorem}
The space of solutions $\mathbb{V}$ is a finite dimensional vector space generated by the 
basis $(\Omega_k)_{0\leq k\leq d-1}$.
\end{theorem}

We recall Carlson's Theorem \cite{Ca}:

\begin{theorem}[Carlson, 1914]
Let $\CC_+ =\{s\in \CC; \Re s>0\}$ and $f:\CC_+ \to \CC$ be a holomorphic function extending continuously
to $\overline{\CC_+}$. We assume that $f$ is of exponential type, that is, there is $C,\tau >0$ such that 
for all $s\in \CC_+$,
$$
|f(s)|\leq Ce^{\tau |s|}
$$
We assume that on the imaginary axes we have a more precise control, for $y\in \RR$,
$$
|f(iy)| \leq Ce^{c|y|}
$$
for some constant $c<\pi$.

If $f(n)=0$ for all $n\in \NN$, then $f$ is identically $0$.
\end{theorem}

We use Carlson's Theorem in the half plane $\{\Re s>1\}$ to prove the main Theorem.

\begin{proof}[\textbf{Proof.}]
We consider a meromorphic solution $f(s)$ of the functional equation and satisfying the estimate (\ref{eq:estimate}). 
The matrix  $[\Omega_{kl}(1)]$ being invertible, we have a linear combination 
$g(s)= c_0 \Omega_0(s)+\ldots +c_{d-1}\Omega_{d-1} (s)$ 
with $c_0, \ldots ,c_{d-1} \in \CC$ such that  $g(l)=f(l)$ for $l=1,2,\ldots , d$. 
Since $g$ satisfies also the functional equation, 
we get by induction using the functional equation that $f$ and $g$ take 
the same values at all the positive integers $s \in \NN^*$. 
So the function $f-g$ vanishes at all positive integers, also by linearity satisfies the functional equation, and therefore
satisfies the estimate (\ref{eq:estimate}). Now, consider the function 
$h(s)=e^{-ic s}(f(s)-g(s))/\Gamma(s/d)$. Because of the factor $e^{-ics}$ and the division by $\Gamma(s/d)$,
from estimate  (\ref{eq:estimate}) we have that $h$ satisfies on $\Re s=1$,
$$
|h(s)|\leq C e^{\frac{\pi}{2} |\Im s|} \ .
$$

Also using the estimate in Proposition \ref{prop:exp_growth}  and  because of the division by 
$\Gamma(s/d)$ in the definition of $h$,  this function $h$ has exponential growth in the right half plane 
$\{\Re s \geq 1\}$. Therefore
using Carlson's Theorem we conclude that $h$ is identically $0$, 
thus $f(s)=g(s)$ for all values $s$ in this half plane, hence in $\CC$.
\end{proof}

We have proved that the vector space generated by Omega functions can be characterized 
by the functional equation (\ref{eq:functional_eq}) and the 
growth property (\ref{eq:estimate}). This generalizes to Omega functions Wielandt's characterization for Euler 
Gamma function (1939, \cite{Wie},  \cite{Rem1}, \cite{Rem2}). 

\bigskip

We also observe that Omega functions provide the general solutions of the functional equation (\ref{eq:functional_eq})
with estimates (\ref{eq:estimate}) since given such a functional equation with coefficients $(\alpha_l)$
we can build the coefficients $a_l=-l^{-1}\alpha_l$, then the polynomial $P_0$ and the Omega functions 
$(\Omega_k)_{0\leq k\leq d-1}$ that form a basis for the space of solutions.

\bigskip

It is also easy to see that we can replace the estimate (\ref{eq:estimate}) by an estimate of the form, for $s\in S(1,b)$,
$$
|f(s)| \leq C e^{-c \Im s}
$$
with $2\pi n\leq c < 2\pi (n+1)$ for an integer $n\in \ZZ$. Then the space of solutions 
is also finite dimensional as the map
$f(s)\mapsto e^{-2\pi i n s} f(s)$ provides an isomorphism of the space of solutions with $\mathbb V$.

\bigskip

The structure of the space of solutions is interesting. The space of holomorphic solutions is a subspace of 
codimension $1$.

\begin{proposition}
The subspace of holomorphic solutions in $\mathbb V$ is a subspace of codimension $1$ generated by the entire functions
of order $1$,
$$
\Omega_{l}(s)-\Omega_0(s) =\int_{\g_{0l}} t^{s-1} e^{P_0(t)}\, dt  \ .
$$
\end{proposition}

\begin{proof}[\textbf{Proof.}]
As observed before, the functions  $\Omega_{l}(s)-\Omega_0(s)$ are entire functions 
of order $1$ and are linearly independent. 
\end{proof}

\bigskip

Some of the results in \cite{BPM3} can be generalized. In particular the Integrability criterion and Abel-like Theorem 
(Theorems 4.2  and 4.3). This will be studied in a separate article. K. Biswas has extended results 
from \cite{BPM3} to curves of higher genus \cite{B}. It is interesting to speculate on the extension of the results 
for Omega functions in higher genus.

\bigskip

\textbf{Acknowledgements.}

We thank J.-P. Ramis for pointing out the earlier appearance of Omega fonctions under the name 
of ``modified Gamma functions'' in the theory of asymptotic approximations, and to the classical 
literature on difference equations, in particular for references on the study of solutions 
with controlled vertical growth. We thank A. Aycock for the reference to Euler's method to solve 
this type of difference equations.
We are grateful to K. Biswas for corrections and interesting 
comments related to our previous  related work. We finally also thank the referee for his comments and corrections
that greatly improved the presentation.


\end{document}